\documentclass{article}
\usepackage{geometry}
\usepackage[british]{babel}
\usepackage{ae}
\usepackage{latexsym}
\usepackage{amsfonts}
\usepackage{amssymb}
\usepackage{amsthm}
\usepackage{graphicx}
\usepackage{multirow}
\usepackage{color}
\usepackage{amsmath}
\usepackage{amscd}
\usepackage{float}
\usepackage{hyperref}

\newcounter{itemcounter}
\numberwithin{itemcounter}{section}

\newtheorem{thm}[itemcounter]{Theorem}
\newtheorem{lem}[itemcounter]{Lemma}
\newtheorem{defi}[itemcounter]{Definition}
\newtheorem{prop}[itemcounter]{Proposition}
\newtheorem{cor}[itemcounter]{Corollary}

\newtheorem{rem}[itemcounter]{Remark}

\newtheorem*{thm*}{Theorem}
\newtheorem*{cor*}{Corollary}

\newcommand{\NN} {\mathbb{N}}

\newcommand{\cO} {\mathcal{O}}

\newcommand{\Irr}{\rm Irr}
\newcommand{\IBr}{\rm IBr}
\newcommand{\Aut}{\rm Aut}
\newcommand{\rk}{\rm rk}
\newcommand{\Out}{\rm Out}

\newcommand{\CF}{\rm CF}
\newcommand{\prj}{\rm prj}
\newcommand{\Ind}{\rm Ind}
\newcommand{\Res}{\rm Res}

\title{Classifying blocks with abelian defect groups of rank $3$ for the prime $2$ \footnote{This research was supported by the EPSRC (grant no. EP/M015548/1).}}
\author{Charles Eaton\footnote{School of Mathematics, University of Manchester, Manchester, M13 9PL, United Kingdom. Email: charles.eaton@manchester.ac.uk} and Michael Livesey\footnote{School of Mathematics, University of Manchester, Manchester, M13 9PL, United Kingdom. Email: michael.livesey@manchester.ac.uk}}
\date{13th October 2017}

\begin{document}
\maketitle

\begin{abstract}
In this paper we classify all blocks with defect group $C_{2^n}\times C_2\times C_2$ up to Morita equivalence. Together with a recent paper of Wu, Zhang and Zhou, this completes the classification of Morita equivalence classes of $2$-blocks with abelian defect groups of rank at most $3$. The classification holds for blocks over a suitable discrete valuation ring as well as for those over an algebraically closed field. The case considered in this paper is significant because it involves comparison of Morita equivalence classes between a group and a normal subgroup of index $2$, so requires novel reduction techniques which we hope will be of wider interest. We note that this also completes the classification of blocks with abelian defect groups of order dividing $16$ up to Morita equivalence. A consequence is that Broue's abelian defect group conjecture holds for all blocks mentioned above.
\end{abstract}

\section{Introduction}

Let $p$ be a prime and $(K,\mathcal{O},k)$ a $p$-modular with $k$ algebraically closed. Let $P$ be a finite $p$-group. Donovan's conjecture states that there are only finitely many Morita equivalence classes amongst blocks of $kG$ for finite groups $G$ with defect groups isomorphic to $P$, and it is natural to strengthen this conjecture to blocks with respect to $\cO$. Advances in our understanding of blocks of finite groups of Lie type in non-defining characteristic open the possibility of using the classification of finite simple groups to tackle this conjecture and further to classify Morita equivalence classes of blocks. For $p=2$ this process has been started in~\cite{ekks14}, where Donovan's conjecture (with respect to $k$) has been proved for all elementary abelian $2$-groups. For elementary abelian $2$-groups of order at most $16$ the Morita equivalence classes have further been classified, with respect to $\cO$ (see~\cite{ea17}). Abelian $p$-groups with a cyclic factor of order strictly great than $p$ present a significant problem to the extension of these results. This is because the case of a group generated by a normal subgroup and a defect group is especially difficult to to study with respect to Morita equivalence, and required the application of~\cite{kk96} which applies only to split extensions by a factor of the defect group, and further only to blocks defined over $k$. In~\cite{eatliv17} a partial generalization of~\cite{kk96} was given (generalized further in Proposition \ref{prop:grunit} below) which was sufficient to work with the Loewy length of blocks with arbitrary abelian defect groups. In this paper we combine this result with the existence of a certain perfect isometry from~\cite{wa05} to prove Donovan's conjecture for blocks (defined over $\cO$) with defect groups $C_{2^n} \times C_2 \times C_2$ for $n>1$ and further show that for each $n$ there are precisely three Morita equivalence classes of such blocks. This completes the classification of Morita equivalence classes of blocks with abelian defect groups of order dividing $16$ (see~\cite{li94},~\cite{ea16} and~\cite{ea17} for the elementary abelian $2$-groups and~\cite{ekks14} for $C_4 \times C_4$, noting that in all other cases $\Aut(D)$ is a $2$-group and so all blocks with that defect group are nilpotent). In~\cite{wzz17} it is shown that for $m,n \in \NN$ with $n \geq 2$, if a block has defect group $D \cong C_{2^n} \times C_{2^n} \times C_{2^m}$, then it is Mortia equivalent to it Brauer correspondent in $N_G(P)$ and so to one of $\cO D$, $\cO (D \rtimes C_3)$, $\cO (D \rtimes C_7)$ or $\cO (D \rtimes (C_7 \rtimes C_3))$. Hence the classification of Morita equivalence classes of $2$-blocks with abelian defect groups of rank at most $3$ is complete.

We refer the reader to~\cite{ea17} for a survey of progress on the problem of classifying Morita equivalence classes of blocks with a given defect group.

Recall the definition of the inertial quotient of a block $B$ of $\cO G$ with defect group $D$, where $G$ is a finite group. Let $b_D$ be a block of $C_G(D)$ with Brauer correspondent $(b_D)^G=B$. The stabilizer of $b_D$ in $N_G(D)$ under conjugation is written $N_G(D,b_D)$. Then $N_G(D,b_D)/DC_G(D)$ is a $p'$-group, and is called the inertial quotient of $B$ (unique up to isomorphism). $|E|$ is called the inertial index. $(D,b_D)$ is called a maximal $B$-subpair.

\begin{lem}
\label{inertial_quotient_lem}
Let $G$ be a finite group and $B$ be a block of $\cO G$ with defect group $D \cong C_{2^n}\times C_2\times C_2$ for $n>1$. There are two possible fusion systems for $B$, given by $D$ and $C_{2^n} \times A_4$. In particular the possible inertial quotients for $B$ are $1$ and $C_3$.
\end{lem}

\begin{proof}
We refer the reader to~\cite{ako} for background on fusion systems. Since $D$ is abelian, the fusion systems on $D$ are given by groups $D \rtimes E$, where $E$ is an odd-order subgroup of $\Aut(D)$. We have $\Aut(D) \cong S_3$, so the possibilities are $E=1$ or $C_3$.
\end{proof}

The main result is as follows (see Theorem \ref{thm:main}):

\begin{thm*}
Let $G$ be a finite group and $B$ a block of $\mathcal{O}G$ with defect group $D\cong C_{2^n}\times C_2\times C_2$ for $n>1$. Then $B$ is Morita equivalent to the principal block of $\mathcal{O}(C_{2^n}\times C_2\times C_2)$, $\mathcal{O}(C_{2^n}\times A_4)$ or $\mathcal{O}(C_{2^n}\times A_5)$.
\end{thm*}

Combining with the results of~\cite{ea16},~\cite{ekks14} and~\cite{wzz17} we conclude:

\begin{cor*}
Let $G$ be a finite group and $B$ a block of $\mathcal{O}G$ with defect group $D$ of $2$-rank at most $3$. Then $B$ is Morita equivalent to the principal block of one of:

(i) $\cO D$;

(ii) $\cO (D \rtimes C_3)$;

(iii) $\cO (C_{2^n}\times A_5)$ for $n \geq 0$;

(iv) $\cO (D \rtimes C_7)$;

(v) $\cO SL_2(8)$;

(vi) $\cO (D \rtimes (C_7 \rtimes C_3))$;

(vii) $\cO J_1$;

(viii) $\cO \Aut(SL_2(8))$.

\end{cor*}

Observe that this means that every block with this defect group is Morita equivalent to a principal block, and so in particular the Morita Frobenius number as defined in~\cite{ke05} is one. Note that if $B$ above is not nilpotent, then the number $l(B)$ of irreducible Brauer characters of $B$ is $3$ and the number of irreducible characters $k(B)=2^{n+2}$ (by Proposition \ref{per_isom}, which does not use the classification of finite simple groups).

Another consequence of Theorem \ref{thm:main} is the following (see Corollary \ref{cor:derived}):

\begin{cor*}
Let $G$ be a finite group and $B$ a block of $\mathcal{O}G$ with defect group $D\cong C_{2^n}\times C_2\times C_2$ for $n>1$. Let $b$ be the unique block of $N_G(D)$ with $b^G=B$. Then $B$ and $b$ are derived equivalent.
\end{cor*}

Gathering together previous results, this completes the proof of Brou\'e's conjecture for $2$-blocks of defect at most $4$ and also for those of rank at most $3$ (see Corollaries \ref{cor:derived16} and \ref{cor:derivedrank3}).

The paper is structured as follows. In Section 2 we address the problem of lifting a Morita equivalence from a normal subgroup of index $2$. We obtain a perfect isometry from~\cite{wa05} and show that this may be modified by a perfect self-isometry to produce a Morita equivalence using the central unit described in~\cite{eatliv17}. In Section 3 we apply the classification in~\cite{ekks14} to prove the main theorem, using the results of Section 2 to help reduce to quasisimple groups.

%%%%%%%%%%%%%%%%%%%%%%%%%%%%%%%%%%%%%%%%%%%%%%%%%%%

\section{Normal subgroups of index 2}

We first introduce some notation.

Let $G$ be a finite group and $N \lhd G$. For a block $B$ of $\cO G$ we write $\Irr(B)$ for the set of irreducible characters of in $B$ (with respect to $K$) and $\Irr(B,\psi)$ for the set of irreducible characters in $B$ covering $\psi$ (that is, whose restriction has $\psi$ as a summand). Write $\IBr(B)$ for the set of irreducible Brauer characters. Write $\prj (B)$ for the set characters of projective indecomposable $B$-modules.

Suppose $B$ has defect group $D$. Let $(D,b_D)$ be a maximal $B$-subpair (see~\cite{alp86} for background on subpairs). Note that all maximal $B$-subpairs are $G$-conjugate. If $u \in D$ and $b_u$ is a block of $\cO C_G(u)$ with $(b_D)^{C_G(u)}=b_u$, then we call $(u,b_u)$ a subsection in $(D,b_D)$, and write $(u,b_u) \in (D,b_D)$.

\begin{defi}[\cite{br90}]
Let $G$ be a finite group and $B$ a block of $\mathcal{O}G$. We denote by $\CF(G,B,K)$ the $K$-subspace of class functions on $G$ spanned by $\Irr(B)$, by $\CF(G,B,\mathcal{O})$ the $\mathcal{O}$-submodule
\begin{align*}
\{\phi\in \CF(G,B,K):\phi(g)\in\mathcal{O}\text{ for all }g\in G\}
\end{align*}
of $\CF(G,B,K)$ and by $\CF_{p'}(G,B,\mathcal{O})$ the $\mathcal{O}$-submodule
\begin{align*}
\{\phi\in \CF(G,B,\mathcal{O}):\phi(g)=0\text{ for all }g\in G\backslash G_{p'}\}
\end{align*}
of $\CF(G,B,\mathcal{O})$.
\newline
\newline
Now in addition let $H$ be a finite group and $C$ a block of $\mathcal{O}H$. A \textbf{perfect isometry} between $B$ and $C$ is an isometry
\begin{align*}
I:\mathbb{Z}\Irr(B)\to\mathbb{Z}\Irr(C),
\end{align*}
such that
\begin{align*}
I_K:=K\otimes_{\mathbb{Z}}I:K\Irr(B)\to K\Irr(C),
\end{align*}
induces an $\cO$-module isomorphism between $\CF(G,B,\mathcal{O})$ and $\CF(H,C,\mathcal{O})$ and also between $\CF_{p'}(G,B,\mathcal{O})$ and $\CF_{p'}(H,C,\mathcal{O})$. (Note that by an isometry we mean an isometry with respect to the usual inner products on $\mathbb{Z}\Irr(B)$ and $\mathbb{Z}\Irr(C)$, so for all $\chi\in\Irr(B)$, $I(\chi)=\pm\psi$ for some $\psi\in\Irr(C)$).
\end{defi}

\begin{rem}
An alternative way of phrasing the condition that $I_K$ induces an isomorphism between $\CF_{p'}(G,B,\mathcal{O})$ and $\CF_{p'}(H,C,\mathcal{O})$ is that $I$ induces an isomorphism $\mathbb{Z}\prj (B) \cong \mathbb{Z}\prj (C)$.
\end{rem}

We will be using the following well-known result frequently, so we include a proof.

\begin{lem}
\label{index_p_background_lem}
Let $G$ be a finite group and $N$ a normal subgroup of $G$ of index a power of $p$. Let $B$ be a block of $\mathcal{O}G$ with defect group $D$ covering a block $b$ of $\cO N$. Then $B$ is the unique block of $\cO G$ covering $b$, $D \cap N$ is a defect group for $b$ and the stabilizer of $b$ in $G$ is $DN$.
\end{lem}

\begin{proof}
That $B$ is the unique block of $\cO G$ covering $b$ is~\cite[V.3.5]{feit}. The rest follows from~\cite[15.1]{alp86}, noting that there is a one to one correspondence between blocks defined over $k$ and blocks defined over $\cO$.
\end{proof}

\begin{prop}\label{prop:grunit}
Let $G$ be a finite group and $N$ a normal subgroup of $G$ of index $p$. Now let $B$ be a block of $\mathcal{O}G$ with abelian defect group $D$ such that $G=ND$. Then there exists a block $b$ of $\mathcal{O}N$ with the same block idempotent as $B$ and defect group $D\cap N$. Moreover there exists a $G/N$-graded unit $a\in Z(B)$, in particular $B=\bigoplus_{j=0}^{p-1}a^jb$.
\end{prop}

\begin{proof}
By~\cite[Theorem 2.1]{eatliv17} there exists a $G/N$-graded unit $\overline{a}\in Z(kB)$. Now as every element of $Z(kG)$ can be lifted to an element of $Z(\mathcal{O}G)$ we can lift $\overline{a}$ to $a\in Z(B)$. As the block idempotent of $B$ lies in $\mathcal{O}N$ and $Z(\mathcal{O}G)$ is $G/N$-graded, then we can assume $a$ is $G/N$-graded. Finally, as $Z(B)$ is a local ring and $a$ certainly does not lie in its maximal ideal ($\overline{a}$ is a unit), we have that $a$ is a unit.
\end{proof}

\begin{rem}\label{rem:inner}
In the setting of the previous Proposition \ref{prop:grunit} we note that conjugation by $G$ induces only inner automorphisms of $b$, and so in particular $Z(b) \subseteq Z(B)$.
\end{rem}

\begin{prop}\label{prop:index_p}
Let $G$, $N$, $B$ and $b$ be as in the Proposition \ref{prop:grunit}. Then

(i) Every irreducible character of $b$ is $G$-stable and extends to $p$ distinct irreducible characters of $B$.

(ii) Induction $\Ind_N^G$ gives a bijection between the projective indecomposable $b$-modules and the projective indecomposable $B$-modules.

(iii) Now let $G'$, $N'$, $B'$ and $b'$ be another quadruple satisfying Proposition \ref{prop:grunit}. For each $\chi \in \Irr(b)$ write $\Irr(B,\chi)= \{\chi_1,\ldots, \chi_p \}$.

Suppose $I:\mathbb{Z}\Irr(B)\to\mathbb{Z}\Irr(B')$ is a perfect isometry such that for each $\chi \in \Irr(b)$ there is $\psi \in \Irr(b')$ and $\epsilon_\chi \in\{\pm1\}$ such that $I(\chi_i)=\epsilon_\chi \psi_i$ for $i=1, \ldots, p$ where $\Irr(B',\psi)= \{\psi_1, \ldots , \psi_p \}$. Then the isometry $I_{N,N'}:\mathbb{Z}\Irr(b)\to\mathbb{Z}\Irr(b')$ defined by $I_{N,N'}(\chi):=\epsilon_\chi \psi$ is perfect.

\end{prop}

\begin{proof}$ $
(i) By Remark~\ref{rem:inner} every character $\chi$ of $b$ is $G$-stable. Therefore, as $G/N$ is cyclic, $\chi$ extends to $G$. Taking the product with the $p$ distinct linear characters of $G/N$ inflated to $G$ gives the $p$ extensions of $\chi$ to $G$. Now every constituent of an irreducible character of $B$ restricted to $N$ must lie in $b$ and so every irreducible character of $B$ is the extension of some irreducible character of $b$.

(ii) Certainly every projective indecomposable $B$-module is a summand of some projective indecomposable $b$-module induced to $G$ and Green's indecomposability ensures $\Ind_N^G$. To prove we have a bijection we need $\Ind^G_N(P)$ and $\Ind^G_N(Q)$ to be non-isomorphic for non-isomorphic projective indecomposable $b$-modules $P$ and $Q$. However, this is true since by~\ref{prop:grunit}, $\Res^G_N\Ind^G_N(P)\cong P^{\oplus p}$.

(iii) For $\chi \in \Irr(b)$ and $\psi \in \Irr(b')$, write $\tilde{\chi}=\sum_{i=1}^p \chi_i$ and $\tilde{\psi}=\sum_{i=1}^p \psi_i$.

Let
\begin{align*}
\phi=\sum_{\chi\in\Irr(b)}\alpha_\chi\chi\in\CF(N,b,\mathcal{O}),
\end{align*}
for some $\alpha_\chi\in K$ and consider
\begin{align*}
\tilde{\phi}:=\sum_{\chi\in \Irr(b)} \alpha_\chi\tilde{\chi}.
\end{align*}
Note that $\tilde{\phi}(g)=p\phi(g)$ for all $g\in N$ and $\tilde{\phi}(g)=0$ for all $g\in G\backslash N$, so $\tilde{\phi}\in p\CF(G,B,\mathcal{O})$. Now note that
\begin{align*}
(I_{N,N'})_K(\phi)=\sum_{\psi\in\Irr(b')}\beta_\psi\psi,
\end{align*}
where
\begin{align*}
I_K(\tilde{\phi})=\sum_{\psi\in\Irr(b')} \beta_\psi\tilde{\psi}.
\end{align*}
As $\tilde{\phi}\in p\CF(G,B,\mathcal{O})$, we have $I_K(\tilde{\phi})\in p\CF(G',B',\mathcal{O})$. Again $I_K(\tilde{\phi})(g)=p(I_{N,N'})_K(\phi)(g)$ for all $g\in N'$ and so $(I_{N,N'})_K(\phi)\in\CF(N',b',\mathcal{O})$. So $(I_{N,N'})_K(\CF(N,b,\mathcal{O}))\subseteq\CF(N',b',\mathcal{O})$ and by an identical argument $(I_{N,N'})_K^{-1}(\CF(N',b',\mathcal{O}))\subseteq\CF(N,b,\mathcal{O})$.
\newline
\newline
Now suppose in addition that $\phi\in\CF_{p'}(N,b,\mathcal{O})$. Then $\tilde{\phi}(g)=p\phi(g)=0$ for all $g\in N\backslash N_{p'}$ and $\tilde{\phi}(g)=0$ for all $g\in G\backslash N$ and so $\tilde{\phi}\in\CF_{p'}(G,B,\mathcal{O})$. Therefore $I_K(\tilde{\phi})\in\CF_{p'}(G',B',\mathcal{O})$ and by the previous paragraph $(I_{N,N'})_K(\phi)\in\CF_{p'}(N',b',\mathcal{O})$ and so $(I_{N,N'})_K$ induces an isomorphism between $\CF_{p'}(N,b,\mathcal{O})$ and $\CF_{p'}(N',b',\mathcal{O})$ and hence it satisfies both the properties of a perfect isometry.
\end{proof}

\begin{lem}\label{lem:isomcent}
Let $G$ and $G'$ be finite groups, $B$ and $B'$ blocks of $\cO G$ and $\cO G'$ respectively and $I:\mathbb{Z}\Irr(B)\to\mathbb{Z}\Irr(B')$ a perfect isometry.

(i) The $K$-algebra isomorphism between $Z(KB)$ and $Z(KB')$ given by the bijection of character idempotents induced by $I$ induces an $\cO$-algebra isomorphism $\phi_I:Z(B)\to Z(B')$.

(ii) Suppose further that $I$ satisfies the conditions of Proposition~\ref{prop:index_p}. Then $\phi_{I_{N,N'}}=\phi_I|_{Z(b)}$.
\end{lem}

\begin{proof}$ $
(i) This is proved in~\cite{br90}.

(ii) Let $\chi\in\Irr(b)$ and $\pm I(\chi)=\psi\in\Irr(b')$. Then, adopting the notation of Proposition~\ref{prop:index_p}, we have that
\begin{align*}
e_\chi=e_{\chi_1}+\dots+e_{\chi_p}\text{ and }e_\psi=e_{\psi_1}+\dots+e_{\psi_p},
\end{align*}
where $e_\varphi$ is the block idempotent of $\varphi$ in the appropriate group algebra over $K$. The statement now follows from the definitions of $I_{N,N'}$ and $\phi_I$.
\end{proof}

Let $n$ be a positive integer. We now work towards constructing all the perfect self-isometries of $\mathcal{O}(C_{2^n}\times A_4)$. These will ultimately be used in Theorem~\ref{index2theorem}. From now on we assume $p=2$. Let $\omega\in K$ be a primitive $3^{\operatorname{rd}}$ root of unity. We recall the character table of $A_4$, where we also set up some labelling of characters.
\begin{align*}
\begin{tabular}{|c||c|c|c|c|} \hline
 & $()$ & $(12)(34)$ & $(123)$ & $(132)$ \\ \hline
$\chi_1$ & $1$ & $1$ & $1$ & $1$ \\
$\chi_2$ & $1$ & $1$ & $\omega$ & $\omega^2$ \\
$\chi_3$ & $1$ & $1$ & $\omega^2$ & $\omega$ \\
$\chi_4$ & $3$ & $-1$ & $0$ & $0$ \\ \hline
\end{tabular}
\end{align*}

\begin{prop}\label{prop:self_A4}
The perfect self-isometries of $\mathcal{O}A_4$ are precisely the isometries of the form:
\begin{align*}
I_{\sigma,\epsilon}:\mathbb{Z}\Irr(A_4)&\to\mathbb{Z}\Irr(A_4)\\
\chi_j&\mapsto\epsilon\delta_j\delta_{\sigma(j)}\chi_{\sigma(j)}
\end{align*}
for $1\leq j\leq 4$, where $\sigma\in S_4$, $\epsilon\in\{\pm1\}$ and $\delta_1=\delta_2=\delta_3=-\delta_4=1$. Hence the group of perfect self-isometries is isomorphic to $C_2 \times S_4$.
\end{prop}

\begin{proof}
We first note that
\begin{align*}
\prj (\mathcal{O}A_4))=\{\chi_{P_1},\chi_{P_2},\chi_{P_3}\},\text{ where }\chi_{P_j}=\chi_j+\chi_4\text{ for }j=1,2,3.
\end{align*}
Therefore
\begin{align*}
\mathbb{Z}\prj (\mathcal{O}A_4))=\left\{\sum_{j=1}^4a_j\chi_j:a_j\in\mathbb{Z},\sum_{j=1}^3a_j=a_4\right\}.
\end{align*}
So the isometries $\mathbb{Z}\Irr(\mathcal{O}A_4)\to\mathbb{Z}\Irr(\mathcal{O}A_4)$ that leave $\mathbb{Z}\prj (\mathcal{O}A_4))$ invariant are precisely the permutations of $\{\chi_1,\chi_2,\chi_3,-\chi_4\}$ together with their negatives. These are precisely the $I_{\sigma,\epsilon}$'s.
\newline
\newline
We now describe $\CF(A_4,\mathcal{O}A_4,\mathcal{O})$. Let $\sum_{j=1}^4a_j\chi_j\in\CF(A_4,\mathcal{O}A_4,K)$. Now by evaluating at various elements of $A_4$ we get that $\sum_{j=1}^4a_j\chi_j\in\CF(A_4,\mathcal{O}A_4,\mathcal{O})$ if and only if
\begin{align*}
a_1+a_2+a_3+3a_4&\in\mathcal{O},\\
a_1+a_2+a_3-a_4&\in\mathcal{O},\\
a_1+\omega a_2+\omega^2a_3&\in\mathcal{O}\\
\text{and }a_1+\omega^2a_2+\omega a_3&\in\mathcal{O}.
\end{align*}
These conditions are equivalent to
\begin{align*}
4a_1\in\mathcal{O},4a_2\in\mathcal{O},4a_3\in\mathcal{O},4a_4\in\mathcal{O}\\
\text{and }\delta_ja_j-\delta_la_l\in\mathcal{O}\text{ for all }1\leq j,l\leq 4.
\end{align*}
One can now check that all the $I_{\sigma,\epsilon}$'s leave $\CF(A_4,\mathcal{O}A_4,\mathcal{O})$ invariant and the proposition is proved.
\end{proof}

Let $\zeta\in K$ be a primitive $(2^{n})^{\operatorname{th}}$ root of unity.

\begin{lem}\label{lem:roots_in_o}
Let $m$ be a positive integer and suppose $\sum_{i=0}^{2^m-1}\zeta^{l_i}\in2^m\mathcal{O}$, where $l_i\in\mathbb{Z}$ for $0\leq i<2^m$. Then either $\zeta^{l_0}=\dots=\zeta^{l_{2^m-1}}$ or $\sum_{i=0}^{2^m-1}\zeta^{l_i}=0$.
\end{lem}

\begin{proof}
We consider $\mathbb{Q}$ as embedded in $K$. First note that
\begin{align*}
\prod_{\sigma\in\operatorname{Gal}(\mathbb{Q}(\zeta)/\mathbb{Q})}\left(\sigma\left(\sum_{i=0}^{2^m-1}\zeta^{l_i}\right)\right)\in2^{m|\operatorname{Gal}(\mathbb{Q}(\zeta)/\mathbb{Q})|}\mathcal{O}\cap\mathbb{Z}=2^{m|\operatorname{Gal}(\mathbb{Q}(\zeta)/\mathbb{Q})|}\mathbb{Z}.
\end{align*}
However, for each $\sigma\in\operatorname{Gal}(\mathbb{Q}(\zeta)/\mathbb{Q})$ we have
\begin{align*}
\left|\sigma\left(\sum_{i=0}^{2^m-1}\zeta^{l_i}\right)\right|\leq2^m
\end{align*}
with equality if and only if $\zeta^{l_0}=\dots=\zeta^{l_{2^m-1}}$, where $|\quad|$ denotes the usual norm in $K$. The claim now follows.
\end{proof}

Let $x$ be a generator of $C_{2^n}$. For $0\leq i<2^n$ we define $\theta_i\in\Irr(C_{2^n})$ by $\theta_i(x)=\zeta^i$.

\begin{prop}\label{prop:self_cyclic}
The perfect self-isometries of $\mathcal{O}C_{2^n}$ are precisely the isometries of the form:
\begin{align*}
I_{j,l,\epsilon}:\mathbb{Z}\Irr(C_{2^n})&\to\mathbb{Z}\Irr(C_{2^n})\\
\theta_i&\mapsto\epsilon\theta_{i+l}^j
\end{align*}
for $0\leq i<2^n$, where $0\leq j,l<2^n$ with $j$ odd, $\epsilon\in\{\pm1\}$ and $i+l$ is considered modulo $2^n$. Moreover each $I_{j,l,1}$ is induced by the $\mathcal{O}$-algebra automorphism $x\mapsto\zeta^{-l}x^{\frac{1}{j}}$. Hence the group of perfect self-isometries is isomorphic to $C_2\times[(\mathbb{Z}/2^n)^\times\ltimes(\mathbb{Z}/2^n)]$, where the action of $(\mathbb{Z}/2^n)^\times$ on $\mathbb{Z}/2^n$ is given by multiplication.
\end{prop}

\begin{proof}
We know $\prj (\mathcal{O}C^{2^n})=\left\{\sum_{i=0}^{2^n-1}\theta_i\right\}$ so any perfect isometry $I:\mathbb{Z}\Irr(\mathcal{O}C_{2^n})\to\mathbb{Z}\Irr(\mathcal{O}C_{2^n})$ must have all signs the same. Therefore we need only check what permutations of the $\theta_i$'s leave $\CF(C_{2^n},\mathcal{O}C_{2^n},\mathcal{O})$ invariant. We first note that
\begin{align*}
\sum_{i=0}^{2^n-1}\zeta^{-i}\theta_i(g)=
\begin{cases}
2^n&\text{ if }g=x,\\
0&\text{ if }g\neq x,
\end{cases}
\end{align*}
for all $g\in C_{2^n}$. So $\frac{1}{2^n}\sum_{i=0}^{2^n-1}\zeta^{-i}\theta_i\in\CF(C_{2^n},\mathcal{O}C_{2^n},\mathcal{O})$. Now consider a perfect isometry induced by $\sigma$, a permutation of $\{0,1,\dots,2^n-1\}$. Then we must have $\frac{1}{2^n}\sum_{i=0}^{2^n-1}\zeta^{-i}\theta_{\sigma(i)}\in\CF(C_{2^n},\mathcal{O}C_{2^n}\mathcal{O})$. So
\begin{align*}
\sum_{i=0}^{2^n-1}\zeta^{-i}\theta_{\sigma(i)}(g)\in2^n\mathcal{O},
\end{align*}
for all $g\in C_{2^n}$. Therefore, by Lemma~\ref{lem:roots_in_o}, we have that for each $g\in C_{2^n}$ that this sum is either zero or all the $\zeta^{-i}\theta_{\sigma(i)}(x)$'s are equal. Certainly they can't all be zero, as we have a non-zero linear combination of characters. % Now suppose that
%\begin{align*}
%&\theta_{\sigma(0)}(x)=\zeta^{-1}\theta_{\sigma(1)}(x)=\dots=\zeta^{-(2^n-1)}\theta_{\sigma(2^n-1)}(x),\\
%&\theta_{\sigma(0)}(y)=\zeta^{-1}\theta_{\sigma(1)}(y)=\dots=\zeta^{-(2^n-1)}\theta_{\sigma(2^n-1)}(y),
%\end{align*}
%for distinct $x,y\in C_{2^n}$. Then
%\begin{align*}
%\theta_{\sigma(0)}(x/y)=\theta_{\sigma(1)}(x/y)=\dots=\theta_{\sigma(2^n-1)}(x/y)
%\end{align*}
%but $x/y\neq1$, a contradiction.
Therefore there exists $y\in C_{2^n}$ such that
%\begin{align*}
%\sum_{i=0}^{2^n-1}\zeta^{-i}\theta_{\sigma(i)}(x)=0,
%\end{align*}
%for all $x\neq y$ and
\begin{align*}
\theta_{\sigma(0)}(y)=\zeta^{-1}\theta_{\sigma(1)}(y)=\dots=\zeta^{-(2^n-1)}\theta_{\sigma(2^n-1)}(y).
\end{align*}
Certainly $y$ must generate $C_{2^n}$ as it takes $2^n$ distinct values on the elements of $C_{2^n}$. Define $0\leq j,l<2^n$ by $x=y^j$ and $\theta_{\sigma(0)}(y)=\zeta^l$ and note that $j$ must be odd. Then
\begin{align*}
\theta^j_{m+l}(x)=\zeta^{j(m+l)}=[\zeta^m\theta_{\sigma(0)}(y)]^j=[\theta_{\sigma(m)}(y)]^j=\theta_{\sigma(m)}(y^j)=\theta_{\sigma(m)}(x),
\end{align*}
for all $0\leq m<2^n$. Therefore, $\theta_{\sigma(m)}=\theta^j_{m+l}$. Finally we note that the isometry $\theta_m\mapsto\theta_{m+l}^j$ is induced by the $\mathcal{O}$-algebra automorphism $x\mapsto\zeta^{-l}x^{\frac{1}{j}}$ and so is indeed a perfect isometry and the proof is complete.
\end{proof}

For the following theorem we adopt the notation of Propositions~\ref{prop:self_A4} and~\ref{prop:self_cyclic}.

\begin{thm}\label{thm:C2nA4}
Every perfect self-isometry of $\mathcal{O}(C_{2^n}\times A_4)$ is of the form
\begin{align*}
(I_{j,l,1},I_{\sigma,\epsilon}):(\mathbb{Z}\Irr(\mathcal{O}C_{2^n})\otimes_{\mathbb{Z}}\mathbb{Z}\Irr(\mathcal{O}A_4))\to(\mathbb{Z}\Irr(\mathcal{O}C_{2^n})\otimes_{\mathbb{Z}}\mathbb{Z}\Irr(\mathcal{O}A_4)),
\end{align*}
where $\sigma\in S_4$, $\epsilon\in\{\pm1\}$ and $0\leq j,l<2^n$ with $j$ odd.
\end{thm}

\begin{proof}
The projective indecomposable characters are
\begin{align*}
\prj (\mathcal{O}(C_{2^n}\times A_4))=\{\chi_{P_1},\chi_{P_2},\chi_{P_3}\},\text{ where }\chi_{P_j}=\left(\sum_{i=0}^{2^n-1}\theta_i\right)\otimes\left(\chi_j+\chi_4\right).
\end{align*}
Let $I$ be a perfect self-isometry of $\mathcal{O}(C_{2^n}\times A_4)$. By counting constituents we see that
\begin{align}\label{align:im}
I(\chi_{P_l})=\pm\chi_{P_1},\pm\chi_{P_2},\pm\chi_{P_3},\pm(\chi_{P_1}-\chi_{P_2}),\pm(\chi_{P_1}-\chi_{P_3})\text{ or }\pm(\chi_{P_2}-\chi_{P_3}),
\end{align}
for $1\leq l\leq 3$. Consider the set
\begin{align*}
X_m:=\left\{j\big{|}\left\langle\theta_l\otimes\chi_j,I\left(\left(\sum_{i=0}^{2^n-1}\theta_i\right)\otimes\chi_m\right)\right\rangle\neq0,\text{ for some }l\right\},
\end{align*}
for $1\leq m\leq 4$. By (\ref{align:im}) we have shown that $|X_m|=1$ or $2$ for every $1\leq m\leq 4$. If $|X_1|=2$, then by considering (\ref{align:im}) for $l=1$ we see that $X_4=X_1$. Similarly by considering $I(\chi_{P_2})$, we get that $X_2=X_4$. This is now a contradiction as then
\begin{align*}
I\left(\left(\sum_{i=0}^{2^n-1}\theta_i\right)\otimes\left(\chi_1+\chi_2+\chi_4\right)\right)
\end{align*}
has at most $2.2^n$ constituents with non-zero multiplicity. Therefore $|X_1|=1$ and so by considering $I(\chi_{P_1})$ we get that $|X_4|=1$ and then by considering $I(\chi_{P_2})$ and $I(\chi_{P_3})$ we get that $|X_2|=|X_3|=1$. Moreover, $X_1,X_2,X_3,X_4$ must all be disjoint. By composing $I$ with the perfect isometry $(I_{1,1,1},I_{\sigma,1})$, for some appropriately chosen $\sigma\in S_4$, we may assume $X_m=\{m\}$ for all $1\leq m\leq4$. Therefore $I(\chi_{P_l})=\pm\chi_{P_l}$ for $1\leq l\leq3$ and by considering
\begin{align*}
I\left(\left(\sum_{i=0}^{2^n-1}\theta_i\right)\otimes\chi_4\right),
\end{align*}
we see that in fact all these signs are the same and we may assume, possibly by composing $I$ with $(I_{1,1,1},I_{1,\epsilon})$, that
\begin{align*}
I\left(\left(\sum_{i=0}^{2^n-1}\theta_i\right)\otimes\chi_m\right)=\left(\sum_{i=0}^{2^n-1}\theta_i\right)\otimes\chi_m,
\end{align*}
for $1\leq m\leq 4$. Next we note that
\begin{align*}
\frac{1}{12}\theta_j\otimes(\chi_1+\chi_2+\chi_3+9\chi_4)\in\CF(C_{2^n}\times A_4,\mathcal{O}(C_{2^n}\times A_4),\mathcal{O}),
\end{align*}
for $0\leq j<2^n$. As $3$ is invertible in $\mathcal{O}$, this implies
\begin{align*}
\theta_j\otimes\left(\sum_{m=1}^4\chi_m\right)\in4\CF(C_{2^n}\times A_4,\mathcal{O}(C_{2^n}\times A_4),\mathcal{O}),
\end{align*}
and so
\begin{align}\label{align:im2}
I\left(\theta_j\otimes\left(\sum_{m=1}^4\chi_m\right)\right)\in4\CF(C_{2^n}\times A_4,\mathcal{O}(C_{2^n}\times A_4),\mathcal{O}),
\end{align}
for $0\leq j<2^n$. Now set $\theta_{j_m}\otimes\chi_m:=I(\theta_j\otimes\chi_m)$, for $1\leq m\leq4$. Evaluating (\ref{align:im2}) at $(x,1)$, $(x,(123))$ and $(x,(132))$ gives
\begin{align}
\zeta^{j_1}+\zeta^{j_2}+\zeta^{j_3}+\zeta^{j_4}&\in4\mathcal{O},\label{zeta1}\\
\zeta^{j_1}+\omega\zeta^{j_2}+\omega^2\zeta^{j_3}&\in4\mathcal{O},\label{zeta2}\\
\zeta^{j_1}+\omega^2\zeta^{j_2}+\omega\zeta^{j_3}&\in4\mathcal{O}\label{zeta3}.
\end{align}
Adding (\ref{zeta1}), (\ref{zeta2}) and (\ref{zeta3}) gives $3\zeta^{j_1}+\zeta^{j_4}\in4\mathcal{O}$. Now Lemma~\ref{lem:roots_in_o} tells us that $\zeta^{j_1}=\zeta^{j_4}$ as certainly $3\zeta^{j_1}+\zeta^{j_4}\neq0$. Therefore, by (\ref{zeta1}), $\zeta^{j_2}+\zeta^{j_3}\in2\mathcal{O}$. So again by Lemma~\ref{lem:roots_in_o} $\zeta^{j_2}=\zeta^{j_3}$ as $\zeta^{j_2}=-\zeta^{j_3}$ is prohibited by (\ref{zeta1}). Substituting into (\ref{zeta2}) gives $\zeta^{j_1}-\zeta^{j_2}\in4\mathcal{O}$. A final use of Lemma~\ref{lem:roots_in_o} tells us that $\zeta^{j_1}=\pm\zeta^{j_2}$ but $2\zeta^{j_1}\notin4\mathcal{O}$ and so we must have $\zeta^{j_1}=\zeta^{j_2}=\zeta^{j_3}=\zeta^{j_4}$.
\newline
\newline
We have shown that we may assume $I$ is of the form
\begin{align*}
I(\theta_j\otimes\chi_m)\mapsto\theta_{\sigma(j)}\otimes\chi_m,
\end{align*}
for $0\leq j<2^n$ and $1\leq m\leq4$, where $\sigma$ is a permutation of $\{0,1,\dots,2^n-1\}$. In particular the $\mathcal{O}$-algebra automorphism of $Z(\mathcal{O}(C_{2^n}\times A_4)$ induced by $I$ leaves $\mathcal{O}C_{2^n}$ invariant. Using Proposition~\ref{prop:self_cyclic} we can compose $I$ with $(I_{s,t,1},I_{1,1})$ for appropriately chosen $s$ and $t$ so that the automorphism induced on $\mathcal{O}C_{2^n}$ is the identity. Therefore $\sigma$ is the identity permutation, $I$ is the identity perfect isometry and the theorem is proved.
\end{proof}

\begin{lem}
\label{normaldefectlemma}
Let $B$ be a block of $\cO G$ for a finite group $G$ with normal defect group $D \cong C_{2^n} \times C_2 \times C_2$ for some $n>1$. Then $B$ is Morita equivalent to $\cO D$ or $\cO(C_{2^n} \times A_4)$.
\end{lem}

\begin{proof}
By Lemma \ref{inertial_quotient_lem} the possible inertial quotients are $1$ and $C_3$. Since in either case the inertial quotient is cyclic the result follows from the main result of~\cite{ku85}.
\end{proof}

%In the next proposition we apply some work of~\cite{wa05} for which we will need some more definitions, which we give here. As usual $G$ is a finite group and $B$ a block of $\cO G$ with abelian defect group $D$. Let $(D,b_D)$ be a maximal $B$-subpair. For $u \in D$, write $b_u$ for the unique block of $C_G(u)$ with $(b_u)^G=B$ and call $(u,b_u)$ a subsection in $(D,b_D)$, writing $(u,b_u) \in (D,b_D)$. Let $\lambda$ be a generalized character of a defect group $D$ of $B$ such that whenever $(u,b_u) \in (D,b_D)$ and $g \in G$ such that $(u,b_u)^g \in (D,b_D)$, we have $\lambda(u)=\lambda(u^g)$. Define $\lambda * \chi$ as in~\cite{bp80}, giving another generalized character of $B$. If $\lambda$ is a generalized character of a factor group of $D$, then we are implicitly considering its inflation to $D$.

The following appears as~\cite[Theorem 15]{sa17}, but we include it here for completeness.

\begin{prop}
\label{per_isom}
Let $G$ be a finite group and $B$ a block of $\cO G$ with defect group $D \cong C_{2^n} \times C_2 \times C_2$ for $n>1$. Then either $B$ is nilpotent or $l(B)=3$. Let $G'=C_{2^n} \times A_4$ and $B'=\cO G'$. If $B$ is not nilpotent then there is a perfect isometry between $B$ and $B'$.

%Further, suppose that additionally $N \lhd G$ and a block $b$ of $\cO N$ so that $G$, $N$, $B$, $b$ and $D$ are as in Proposition \ref{prop:grunit}, and $N'=C_{2^{n-1}} \times A_4$, $b'=\cO N'$, then the perfect isometry may be chosen to satisfy the conditions in Proposition \ref{prop:index_p}(iii).
\end{prop}

\begin{proof}
Write $E$ for the inertial quotient of $B$. If $E=1$, then $B$ is nilpotent and the result holds. Hence we may assume $|E|=3$. We must first show that $l(B)=3$, and we do this by adapting a method used in~\cite{ks13}. We proceed by induction on $n$. Assume that $l(B')=3$ whenever $B'$ is a block with defect group $C_{2^m} \times C_2 \times C_2$ for $m \geq 1$ and inertial index $3$, and observe that this is known to hold for $m=1$ by~\cite{kkl12}.

By~\cite{km13} every irreducible character of $B$ has height zero, and so by~\cite[1.2(ii)]{ro92} we have $k(B) \leq |D|$. Let $(D,b_D)$ be a maximal $B$-subpair. Since $D$ is abelian, $N_G(D,b_D)$ controls fusion of $B$-subpairs in $(D,b_D)$. If $u \in D$, then let $b_u$ be the unique block of $C_G(u)$ such that $(u,b_u) \in (D,b_D)$. Write $D=\langle x,y_1,y_2 \rangle$, where $y$ has order $2^n$ and $y_1,y_2$ are involutions. Then $\mathcal{X}:=\{ (1,B),(x^i,b_{x^i}),(x^jy_1,b_{x^jy_1}):1 \leq i \leq 2^n-1, 0 \leq j \leq 2^n-1 \}$ form a complete set of $G$-conjugacy class representatives of subsections in $B$. By a well known reformulation of~\cite[5.12]{na98} (see exercise 5.7 of~\cite{na98}) we then have $$2^{n+2} \geq k(B) =\sum_{(u,b_u) \in \mathcal{X}} l(b_u).$$ Now since $D$ is abelian, each block $b$ in the above summation may be chosen to have defect group $D$. First let $1 \neq u=x^i$ for some $i$. Then $b_u$ has inertial index $3$. Now $C_G(u)$ has a non-trivial central $2$-subgroup $Z_u \leq \langle x \rangle$. The unique block $\bar{b}_u$ of $C_G(u)/Z_u$ corresponding to $b_u$ has inertial index $3$ and by induction $3=l(\bar{b}_u) = l(b_u)$. Now let $u=x_jy_1$ for some $j$. Then $b_u$ has inertial index $1$ and so is nilpotent, hence $l(b_u)=1$. Substituting, we have $l(B) + 2^{n+2}-3 = k(B) \leq 2^{n+2}$, so $l(B) \leq 3$.

We have a subsection $(u,b_u)$ with defect group $D$ and $l(b_u)=1$. By~\cite[1.37]{sambale}, the diagonal entries of the contribution matrix of $B$ (with rows labelled by $\Irr(B)$) are odd squares, and the trace of the contribution matrix is $|D|$. Hence $2^{n+2}$ is a sum of $k(B)$ odd squares, which cannot happen if $k(B)=2^{n+2}-1$ or $k(B)=2^{n+2}-2$. Hence $l(B)=3$.

Let $C$ be the unique block of $N_G(D)$ with $C^G=B$. By~\cite{wa05} there is a perfect isometry between $B$ and $C'$. By Lemma \ref{normaldefectlemma} $C$ and $B'$ are Morita equivalent, so there is a perfect isometry between $B'$ and $C$, and we are done.
\end{proof}

\begin{rem}
In the above, the perfect isometry constructed in~\cite{wa05} is additionally compatible with the $*$ construction in~\cite{bp80} and so could be shown to satisfy the hypotheses of Proposition~\ref{prop:index_p}(iii). However this can also be shown using the machinery of perfect self-isometry groups developed earlier in this section, and this is what we do in the first part of the proof of the following Theorem.
\end{rem}

\begin{thm}
\label{index2theorem}
Let $G$, $N$, $B$, $b$ and $D$ be as in Proposition~\ref{prop:grunit}. Suppose further that $D\cong C_{2^n}\times C_2\times C_2$, for some $n>1$, $D\cap N\cong C_{2^{n-1}}\times C_2\times C_2$ and $b$ is Morita equivalent to the principal block of $\mathcal{O}(C_{2^{n-1}}\times A_4)$ (respectively $\mathcal{O}(C_{2^{n-1}}\times A_5)$). Then $B$ is Morita equivalent to the principal block of $\mathcal{O}(C_{2^n}\times A_4)$ (respectively $\mathcal{O}(C_{2^n}\times A_5)$).
\end{thm}

\begin{proof}
First suppose that $b$ is Morita equivalent to $\mathcal{O}(C_{2^{n-1}}\times A_4)$.

By Proposition~\ref{prop:index_p}(ii) $l(B) = 3$ and so by Proposition \ref{per_isom} there exists a perfect isometry
\begin{align*}
I:\mathbb{Z}\Irr(B)\to\mathbb{Z}\Irr(\mathcal{O}(C_{2^n}\times A_4)).
\end{align*}
Now $I$ induces an isomorphism of the groups of perfect self-isometries of $B$ and of $\mathcal{O}(C_{2^n}\times A_4)$ via $\alpha \mapsto I \circ \alpha \circ I^{-1}$ for $\alpha$ any perfect self-isometry of $\mathcal{O}(C_{2^n}\times A_4)$, and we denote this isomorphism by $I_{\operatorname{PI}}$. Consider the perfect self-isometry
\begin{align*}
J:\mathbb{Z}\Irr(B)&\to\mathbb{Z}\Irr(B)\\
\chi&\mapsto\operatorname{sgn}_N^G.\chi,
\end{align*}
where $\operatorname{sgn}_N^G$ is the linear character of $G$ with kernel $N$, so for each irreducible character $\theta$ of $b$, $J$ swaps the two extensions of $\theta$ to $G$. We know that $J$ is indeed a perfect isometry as it is induced by the $\mathcal{O}$-algebra automorphism of $\mathcal{O}G$ given by $g\mapsto\operatorname{sgn}_N^G(g)g$ for all $g\in G$.
\newline
\newline
Note that $J$ is a perfect self-isometry of order $2$ and that it induces the trivial $k$-algebra automorphism on $Z(kB)$. Furthermore by Proposition~\ref{prop:index_p}(i) and (ii) every character in $\prj (B)$ is fixed under multiplication by $\operatorname{sgn}_N^G$ and so $J$ is the identity on $\mathbb{Z}\prj (B)$. Therefore $I_{\operatorname{PI}}(J)$ must be of order $2$, induce the identity $k$-algebra automorphism on $Z(k(C_{2^n}\times A_4))$ and be the identity on $\mathbb{Z}\prj (\mathcal{O}(C_{2^n}\times A_4))$.
\newline
\newline
Adopting the notation of Theorem \ref{thm:C2nA4}, set $I_{\operatorname{PI}}(J)=(I_{j,l,1},I_{\sigma,\epsilon})$, where $\sigma\in S_4$, $\epsilon\in\{\pm1\}$ and $0\leq j,l<2^n$ with $j$ odd. Then the fact that $I_{\operatorname{PI}}(J)$ is the identity on $\mathbb{Z}\prj (\mathcal{O}(C_{2^n}\times A_4))$ forces $\sigma$ to be the identity permutation and $\epsilon=1$, the fact that $I_{\operatorname{PI}}(J)$ induces the identity $k$-algebra automorphism on $Z(k(C_{2^n}\times A_4))$ forces $j=1$ and the fact that $I_{\operatorname{PI}}(J)$ has order $2$ forces $l=2^{n-1}$. In other words $I_{\operatorname{PI}}(J)$ is induced by the $\mathcal{O}$-algebra automorphism
\begin{align*}
\mathcal{O}(C_{2^n}\times A_4)&\to\mathcal{O}(C_{2^n}\times A_4)\\
x\otimes y&\mapsto -x\otimes y,
\end{align*}
for all $y\in\mathcal{O}A_4$, where $x$ is a fixed generator of $C_{2^n}$. We have shown that
\begin{align*}
I(\operatorname{sgn}_N^G.\chi)=\operatorname{sgn}_{N'}^{G'}.I(\chi),
\end{align*}
for all $\chi\in\Irr(B)$, where $G':=C_{2^n}\times A_4$, $N':=C_{2^{n-1}}\times A_4$. Therefore $I$ satisfies the hypotheses of Proposition~\ref{prop:index_p}(iii), where $B':=\mathcal{O}(C_{2^n}\times A_4)$ and $b':=\mathcal{O}(C_{2^{n-1}}\times A_4)$. Let $I_{N,N'}$ be the perfect isometry between $b$ and $b'$ induced by $I$ as in Proposition \ref{prop:index_p} and $I_{\operatorname{Mor}}$ the perfect isometry induced by the Morita equivalence between $b$ and $b'$. Write $I_{N,N'}\circ I_{\operatorname{Mor}}^{-1}=(I_{s,t,1},I_{\tau,\delta})$ in the notation of Theorem \ref{thm:C2nA4} applied to $\mathcal{O}(C_{2^{n-1}}\times A_4)$, where $\tau\in S_4$, $\delta\in\{\pm1\}$ and $0\leq s,t<2^{n-1}$ with $s$ odd. By composing $I$ with the perfect self-isometry $(I_{1,1,1},I_{\tau,\delta})^{-1}$ of $B'$ and composing the Morita equivalence $b\sim_{\operatorname{Mor}}b'$ with that induced by the $\mathcal{O}$-algebra automorphism of $b'$ defined by $x\mapsto\zeta^{-t}x^{\frac{1}{s}}$, we may assume that $I_{N,N'}=I_{\operatorname{Mor}}$.
\newline
\newline
Let $\phi_I:Z(B)\to Z(B')$ be the isomorphism of centres from Lemma~\ref{lem:isomcent} and let $M$ be the $b'$-$b$-bimodule inducing the Morita equivalence $b\sim_{\operatorname{Mor}}b'$. Since $I_{N,N'}=I_{\operatorname{Mor}}$ and by Lemma~\ref{lem:isomcent}(ii) we have that $\phi_I|_{Z(b)}=\phi_{I_{N,N'}}:Z(b)\to Z(b')$ is the isomorphism of centres induced by the Morita equivalence. In other words
\begin{align}\label{centre:mor}
\phi_I(\alpha)m=m\alpha,\text{ for all }\alpha\in b,m\in M.
\end{align}
Let $a\in B$ be a graded unit as described in Proposition~\ref{prop:grunit} and set $a':=\phi_I(a)$. Since $\phi_I$ respects the $G/N$ and $G'/N'$-gradings, $a'$ is also a graded unit. We now give $M$ the structure of a module for
\begin{align*}
(b'\otimes_{\mathcal{O}}b^{\operatorname{op}})\oplus(a'^{-1}b'\otimes_{\mathcal{O}}(ab)^{\operatorname{op}})
\end{align*}
by defining $a'^{-1}.m.a=m$, for all $m\in M$, where \ref{centre:mor} ensures that this does indeed define a module. Now by~\cite[Theorem 3.4]{mar96} we have proved that $B$ is Morita equivalent to $\mathcal{O}(C_{2^n}\times A_4)$.
\newline
\newline
For the $A_5$ case we note that the principal blocks of $\mathcal{O}A_4$ and $\mathcal{O}A_5$ are perfectly isometric by~\cite[A1.3]{br90}. The proof now proceeds exactly as above by replacing the principal block of $\mathcal{O}A_4$ everywhere with that of $\mathcal{O}A_5$ (note that we can replace the principal block of $\mathcal{O}A_4$ with that of $\mathcal{O}A_5$ in Theorem \ref{thm:C2nA4}).
\end{proof}

%%%%%%%%%%%%%%%%%%%%%%%%%%%%%%%%%%%%%%%%%%%%

\section{Proof of the main theorem and corollaries}

%\begin{lem}
%Let $B$ be a block of $\cO G$ for a finite group $G$ with normal defect group $D \cong C_{2^n} \times C_2 \times C_2$ for some $n>1$. Suppose that $N \lhd G$ with $[G:N]=2$ and let $b$ be a $G$-stable block of $\cO N$ covered by $B$. Then $b$ has inertial quotient $C_3$.
%\end{lem}

\begin{prop}
\label{oddindex}
Let $B$ be a block of $\cO G$ for a finite group $G$ with defect group $D \cong C_{2^n} \times C_2 \times C_2$ for some $n>1$. Let $N \lhd G$ be of odd prime index $w$ and let $b$ be a $G$-stable block of $\cO N$ covered by $B$, so that $D$ is also a defect group for $b$. If $b$ is not nilpotent, then either $B$ is nilpotent or $B \sim_{\operatorname{Mor}} b$.
\end{prop}

\begin{proof}
By Proposition \ref{per_isom} we have $l(b)=3$ and either $B$ is nilpotent (with $l(B)=1$) or $l(B)=3$. The normal subgroup $G[b]$ of $G$ is defined to be the group of elements of $G$ acting as inner automorphisms on $b \otimes_{\cO} k$. Let $B'$ be a block of $\cO G[b]$ covered by $B$. Then $b$ is source algebra equivalent to $B'$, and in particular has isomorphic inertial quotient by~\cite[2.2]{kkl12}, noting that a source algebra equivalence over $k$ implies one over $\cO$ by~\cite[7.8]{pu88}. Hence we may assume that $G[b]=N$. Then $B$ is the unique block of $G$ covering $B'$ by~\cite[3.5]{da73}.

Now consider the action of $G$ on the $\IBr(b)$. If $w>3$, then every $\varphi \in \IBr(b)$ is fixed and extends to $w$ distinct elements of $\IBr(G)$. Since $B$ is the unique block of $G$ covering $b$ these all lie in $B$. Hence $l(B)=3w>3$, a contradiction. Hence $w=3$. Then either every element of $\IBr(b)$ is fixed, in which case $l(B)=3w>3$, a contradiction, or they are permuted in a single orbit, in which case $l(B)=1$ and $B$ is nilpotent.
\end{proof}

In the following write $\rk_p(Q)$ for the rank of a $p$-group $Q$, that is, $p^{\rk_p(Q)}$ is the size of the largest elementary abelian subgroup of $Q$.

\begin{thm}
\label{thm:main}
Let $G$ be a finite group and $B$ a block of $\mathcal{O}G$ with defect group $D\cong C_{2^n}\times C_2\times C_2$ for $n>1$. Then $B$ is Morita equivalent to the principal block of $\mathcal{O}(C_{2^n}\times C_2\times C_2)$, $\mathcal{O}(C_{2^n}\times A_4)$ or $\mathcal{O}(C_{2^n}\times A_5)$.
\end{thm}

\begin{proof}
Let $B$ be a block of $\cO G$ for a finite group $G$ with $[G:O_{2'}(Z(G))$ minimised subject to the condition that $B$ has defect group $D \cong C_{2^n}\times C_2\times C_2$ for some $n>1$ and $B$ is not Morita equivalent to the principal block of $\cO D$, $\cO(C_{2^n}\times C_2\times C_2)$, $\mathcal{O}(C_{2^n}\times A_4)$ or $\mathcal{O}(C_{2^n}\times A_5)$.

Suppose that $N \lhd G$ and $b$ is a block of $\cO N$ covered by $B$. Write $I=I_G(b)$ for the stabilizer of $b$ in $G$, and $B_I$ for the Fong-Reynolds correspondent. Now $B_I$ is Morita equivalent to $B$ and they have isomorphic defect groups. We have $O_{2'}(Z(G)) \leq O_{2'}(Z(I))$, and if $I \neq G$, then $[I:O_{2'}(Z(I))]<[G:O_{2'}(Z(G))]$. Hence by minimality $I=G$.

Now suppose that $b$ is nilpotent. Let $b'$ be a block of $\cO Z(G)N$ covered by $B$ and covering $b$. By the above argument applied to $Z(G)N$ and $b'$, $b'$ is $G$-stable. Note that $b'$ must also be nilpotent. Using the results of~\cite{kp90}, as outlined in~\cite[Proposition 2.2]{ekks14}, $B$ is Morita equivalent to a block $\tilde{B}$ of a central extension $\tilde{L}$ of a finite group $L$ by a $2'$-group such that there is an $M \lhd L$ with $M \cong D \cap (Z(G)N)$, $G/Z(G)N \cong L/M$, and $\tilde{B}$ has defect group isomorphic to $D$. Note that $[\tilde{L}:O_{2'}(Z(\tilde{L}))] \leq |L| = [G:Z(G)N]|D \cap (Z(G)N)| \leq [G:O_{2'}(Z(G))]$ and that equality only occurs when $N \leq Z(G)O_2(G)$. Hence by minimality $N \leq Z(G)O_2(G)$.

We conclude that $B$ is quasiprimitive, that is, every block of every normal subgroup covered by $B$ is $G$-stable, and that if $B$ covers a nilpotent block of a normal subgroup $N$ of $G$, then $N \leq Z(G)O_2(G)$.

We claim that $O^2(G)=G$. Suppose otherwise, and let $N \lhd G$ be a subgroup of index $2$. Let $b$ be the unique block of $N$ covered by $B$. Then by Lemma \ref{index_p_background_lem} $B$ is the unique block of $G$ covering $b$ since $G/N$ is a $2$-group, $G=ND$ and $b$ has defect group $D \cap N$. Let $b_D$ be a block of $C_G(D)$ with $(b_D)^G=B$. Since $B$ has inertial quotient $C_3$ and $N_G(D,b_D)$ controls fusion in $D$, the inertial quotient of $b$ is $C_3$ (if it were $1$, then $b$ would be nilpotent and so $G=Z(G)O_2(G)$, a contradiction by Lemma \ref{normaldefectlemma}). If $D \cap N \cong C_{2^n} \times C_2$, then $\Aut(D \cap N)$ is a $2$-group and so $b$ is nilpotent, a contradiction. If $D \cap N \cong (C_2)^3$, then by~\cite{ea16} $b$ is Morita equivalent to the principal block of $\cO(C_2 \times A_4)$ or $\cO(C_2 \times A_5)$ and so Theorem \ref{index2theorem} gives a contradiction. Otherwise, since $[N:O_{2'}(Z(N))]<[G:O_{2'}(Z(G))]$ by minimality we also have a contradiction by Theorem \ref{index2theorem}. Hence $O^2(G)=G$.

Before proceeding we recall the definition and some properties of the generalized Fitting subgroup $F^*(G)$ of a finite group $G$. Details may be found in~\cite{asc00}. A \emph{component} of $G$ is a subnormal quasisimple subgroup of $G$. The components of $G$ commute, and we define the \emph{layer} $E(G)$ of $G$ to be the normal subgroup of $G$ generated by the components. It is a central product of the components. The \emph{Fitting subgroup} $F(G)$ is the largest nilpotent normal subgroup of $G$, and this is the direct product of $O_r(G)$ for all primes $r$ dividing $|G|$. The \emph{generalized Fitting subgroup} $F^*(G)$ is $E(G)F(G)$. A crucial property of $F^*(G)$ is that $C_G(F^*(G)) \leq F^*(G)$, so in particular $G/F^*(G)$ may be viewed as a subgroup of $\Out(F^*(G))$.

Write $L_1,\ldots,L_t$ for the components of $G$, so $E(G)=L_1\cdots L_t \lhd G$. Note that $G$ permutes the $L_i$. There must be at least one component, since otherwise the block $b'$ of $F^*(G)$ covered by $B$ is nilpotent and so $F^*(G)=Z(G)O_2(G)$. Therefore $D \leq C_G(F^*(G)) \leq F^*(G)=Z(G)O_2(G)$, so that $D \lhd G$, a contradiction by Lemma \ref{normaldefectlemma}.

We claim that $O_2(G) \leq Z(G)$. Write $N=C_G(O_2(G))$ and $b$ for the unique block of $N$ covered by $B$. Note that $D \leq N \lhd G$. If $O_2(G) \cong C_{2^m}$ for $m \geq 1$ or $O_2(G) \cong C_{2^m} \times C_2$ for $m>1$, then $\Aut(O_2(G))$ and so $G/N$ is a $2$-group, which forces $N=G$ as $O^2(G)=G$. Suppose that $O_2(G) \cong C_{2^m} \times (C_2)^2$ for some $m \geq 1$. Let $b_D$ be a block of $C_G(D)$ with $(b_D)^{C_G(D)}=b$ (and so $(b_D)^G=B$). Since $N_G(D,b_D)$ controls fusion in $D$ there must be a $G$-stable subgroup of $O_2(G)$ of order $4$. Hence $G/N$ is isomorphic to a subgroup of $S_3$. Since $O^2(G)=G$, we have $[G:N]|3$. Then by Proposition \ref{oddindex} and minimality (noting that $Z(G)\leq N$) we again have $G=N$, so $O_2(G) \leq Z(G)$ as claimed.

We have shown that $F^*(G)=E(G)Z(G)$. We next show that $t=1$, that is, $E(G)$ is quasisimple. Write $b^*$ for the unique block of $F^*(G)$ covered by $B$. Then $D \cap F^*(G)$ is a defect group for $b^*$. Hence $(D \cap F^*(G))/O_p(Z(G))$ is a defect group for a block of $F^*(G)/O_p(Z(G))$. Therefore $(D \cap F^*(G))/O_p(Z(G))$ is a radical $2$-subgroup of $F^*(G)/O_p(Z(G))$ (recall that a $p$-subgroup $Q$ of a finite group $H$ is radical if $Q=O_p(N_H(Q))$ and that defect groups are radical $p$-subgroups) and so $(D \cap F^*(G))Z(G)/Z(G)$ is a radical $2$-subgroup of $F^*(G)/Z(G)\cong (L_1Z(G)/Z(G)) \times \cdots \times (L_tZ(G)/Z(G))$. By~\cite[Lemma 2.2]{ou95} it follows that $(D \cap F^*(G))Z(G)/Z(G) = D_1 \times \cdots \times D_m$, where $D_i = (D \cap F^*(G))Z(G)/Z(G)) \cap (L_iZ(G)/Z(G))$ (and $D_i$ is a radical $2$-subgroup but not necessarily a defect group). Write $b_i$ for the block of $L_i$ covered by $B$ and $\bar{b}_i$ for the unique block of $L_iO_2(G)/O_2(G)$ corresponding to $b_i$. If $\rk_p(D_i)=1$ for some $i$, then $\bar{b}_i$ has cyclic defect group and so is nilpotent, hence $b_i$ is also nilpotent by~\cite{wa94} (where the result is stated over $k$, but follows over $\cO$ immediately), a contradiction. Hence since $\sum_{i=1}^t \rk_p(D_i) \leq \rk_p(D) = 3$ we have $t=1$.

Now by the Schreier conjecture $G/F^*(G)$ is solvable. Suppose that $G \neq F^*(G)$. Since $O^2(G)=G$ there is $N \lhd G$ of odd prime index. Let $b$ be the unique block of $N$ covered by $G$. Note that we may assume $Z(G) \leq N$, as otherwise $B$ and $b$ are Morita equivalent by~\cite[2.2]{kkl12} and we may replace $G$ and $B$ by $N$ and $b$. Therefore we have $[N:O_{2'}(Z(N))] < [G:O_{2'}(Z(G))]$ and so Proposition \ref{oddindex} leads to a contradiction. Hence we may assume $G=F^*(G)$ and so $G=L_1Z(G)$. Further application of Proposition \ref{oddindex} and Theorem \ref{index2theorem} allows us to assume that $G=L_1$. Applying~\cite[6.1]{ekks14}, one of the following occurs, both leading to a contradiction, and we are done:

(i) $B$ is Morita equivalent to a block $C$ of $\cO H$ for a finite group $H$ with $H=H_0 \times H_1$ such that $H_0$ is abelian with Sylow $2$-subgroup $C_{2^n}$ and the block of $H_1$ covered by $C$ has defect groups $C_2 \times C_2$. In this case it follows from~\cite{li94} that $B$ is Morita equivalent to the principal block of $\cO D$, $\cO(C_{2^n} \times A_4)$ or $\cO(C_{2^n} \times A_5)$, a contradiction.

(ii) There is a finite group $H$ with $G \lhd H$ and $B$ is covered by a nilpotent block of $\cO H$. In this case by~\cite[4.3]{pu11} $B$ is Morita equivalent to the unique block of $N_G(D)$ with Brauer correspondent $B$, a contradiction by Lemma \ref{normaldefectlemma}.
\end{proof}

\begin{cor}
\label{cor:derived}
Let $G$ be a finite group and $B$ a block of $\mathcal{O}G$ with defect group $D\cong C_{2^n}\times C_2\times C_2$ for $n>1$. Let $b$ be the unique block of $N_G(D)$ with $b^G=B$. Then $B$ and $b$ are derived equivalent.
\end{cor}

\begin{proof}
By~\cite[\textsection 3]{ri96} the principal blocks of $\cO A_4$ and $\cO A_5$ are derived equivalent, and so the same is true of $\cO (C_{2^n} \times A_4)$ and $\cO (C_{2^n} \times A_5)$. Hence by Theorem \ref{thm:main} there are only two derived equivalence classes of blocks with defect group $C_{2^n}\times C_2\times C_2$, and we are done.
\end{proof}

\begin{cor}
\label{cor:derived16}
Let $G$ be a finite group and $B$ a block of $\mathcal{O}G$ with defect group $D$ of order dividing $16$. Let $b$ be the unique block of $N_G(D)$ with $b^G=B$. Then $B$ and $b$ are derived equivalent.
\end{cor}

\begin{proof}
If $D$ is elementary abelian, then this is by~\cite{li94},~\cite{ea16} and~\cite{ea17}. If $D \cong C_4 \times C_4$, then see~\cite{ekks14} where it is shown that there are only two Morita equivalence classes. If $D \cong C_4 \times C_2 \times C_2$, then this is Corollary \ref{cor:derived}. In all other cases $\Aut(D)$ is a $2$-group and so all blocks with that defect group are nilpotent.
\end{proof}

\begin{cor}
\label{cor:derivedrank3}
Let $G$ be a finite group and $B$ a block of $\mathcal{O}G$ with defect group $D$ of $2$-rank at most three. Let $b$ be the unique block of $N_G(D)$ with $b^G=B$. Then $B$ and $b$ are derived equivalent.
\end{cor}

\begin{proof}
This follows immediately from the above corollaries,~\cite[1.1]{ekks14} and~\cite{wzz17}.
\end{proof}

\end{document}